\documentclass[leqno,12pt]{amsart}

\usepackage{amssymb} 
\usepackage{amsmath}
\usepackage{amsthm}
\usepackage{amscd}
\usepackage{bm}
\usepackage{graphicx,psfrag,wrapfig}

\theoremstyle{plain} 
\newtheorem{theorem}{Theorem}[section] 
\newtheorem{lemma}[theorem]{Lemma}
\newtheorem{corollary}[theorem]{Corollary}

\theoremstyle{definition} 

\newtheorem{remark}[theorem]{Remark}
\newtheorem{example}[theorem]{Example}

\newcommand{\affine}{\mathbb{C}}

\newcommand{\vol}{\mathrm{vol}}

\newcommand{\moduli}{\mathcal{M}}

\title{Remark on energy density of Brody curves}
\author[M. Tsukamoto]{Masaki Tsukamoto}



\date{\today}

\begin{document}

\maketitle

\begin{abstract}
We introduce several definitions of energy density of Brody curves and show that 
they give the same value in an appropriate situation.
\end{abstract}

\section{Introduction.} \label{Introduction}
Let $z=x+y\sqrt{-1}\in \affine$ be the standard coordinate of the complex plane $\affine$.
Let $X$ be a compact Hermitian manifold with the K\"{a}hler form $\omega$, 
and let $f:\affine \to X$ be a holomorphic map.
We define the spherical derivative $|df|(z)\geq 0$ by 
\[ f^*\omega = |df|^2 dxdy.\]
We call $f$ a Brody curve (cf. Brody \cite{Brody}) if it satisfies $|df|(z)\leq 1$
for all $z\in \affine$.
Let $\moduli(X)$ be the space of Brody curves in $X$.
This is equipped with the compact-open topology, and it becomes a compact metrizable space
(possibly infinite dimensional) with the following natural continuous $\affine$-action:
\[ \affine \times \moduli(X)\to \moduli(X), \quad (a, f(z))\mapsto f(z+a).\]

For $f\in \moduli(X)$, we define the energy density $\rho(f)$ (first introduced in \cite{Matsuo-Tsukamoto}) by 
\[ \rho(f) := \lim_{R\to \infty} \left(\frac{1}{\pi R^2}\sup_{a\in \affine}\int_{|z-a|<R}|df|^2 dxdy\right).\]
(This limit always exists by Lemma \ref{lemma: Ornstein-Weiss} in Section \ref{section: technical result}.)
Let $\mathcal{N}\subset \moduli(X)$ be a $\affine$-invariant closed subset.
We define $\rho(\mathcal{N})$ as the supremum of $\rho(f)$ over all $f\in \mathcal{N}$.
We sometimes denote $\rho(\moduli(X))$ by $\rho(X)$.

The idea of introducing $\rho(\mathcal{N})$ began in the paper \cite{Tsukamoto-packing}.
(\cite{Tsukamoto-packing} uses a different definition.)
It has a close relation to the mean dimension theory 
(introduced by Gromov \cite{Gromov}).
The paper \cite{Matsuo-Tsukamoto} proves 
\[ 2(N+1)\rho(\affine P^N) \leq \dim(\moduli(\affine P^N):\affine) \leq 4N\rho(\affine P^N).\]
Here $\affine P^N$ is the projective space with the standard Fubini-Study metric, and 
$\dim(\moduli(\affine P^N):\affine)$ is the mean dimension of $\moduli(\affine P^N)$.
In particular 
\[ \dim(\moduli(\affine P^1):\affine) = 4\rho(\affine P^1).\]

The purpose of the present paper is to study variants of $\rho(\mathcal{N})$ and to show that 
they give the same value.

Let $T(r, f)$ be the Nevanlinna-Shimizu-Ahlfors characteristic function of $f\in \moduli(X)$:
\[ T(r, f) := \int_1^r\left(\int_{|z|<t}|df|^2 dxdy\right)\frac{dt}{t} \quad (r\geq 1).\]
Since $|df|\leq 1$ we have $T(r,f)\leq \pi r^2/2$.
We define $\rho_{\mathrm{NSA}}(f)$ and $\underline{\rho}_{\mathrm{NSA}}(f)$ by 
\begin{equation*}
 \begin{split}
   \rho_{\mathrm{NSA}}(f) &:= \limsup_{r\to \infty}\frac{2}{\pi r^2}T(r,f), \\
   \underline{\rho}_{\mathrm{NSA}}(f) &:= \liminf_{r\to \infty}\frac{2}{\pi r^2}T(r,f).
 \end{split}
\end{equation*}
For a $\affine$-invariant closed subset $\mathcal{N}\subset \moduli(X)$, let 
$\rho_{\mathrm{NSA}}(\mathcal{N})$ and $\underline{\rho}_{\mathrm{NSA}}(\mathcal{N})$ 
be the supremums of $\rho_{\mathrm{NSA}}(f)$ and $\underline{\rho}_{\mathrm{NSA}}(f)$
over $f\in \mathcal{N}$ respectively.
It is easy to see $\underline{\rho}_{\mathrm{NSA}}(f)\leq \rho_{\mathrm{NSA}}(f) \leq \rho(f)$.
Hence $\underline{\rho}_{\mathrm{NSA}}(\mathcal{N})\leq \rho_{\mathrm{NSA}}(\mathcal{N})
\leq \rho(\mathcal{N})$.

The quantity $\rho_{\mathrm{NSA}}(\moduli(X))$ 
naturally appeared in the study of the upper bound on the mean dimension \cite{Tsukamoto-Nagoya}.
\begin{example}
Consider $\mathbb{Z}^2 = \{(x,y)|\, x,y\in \mathbb{Z}\}\subset \affine$.
Let $a_n$ $(n\geq 1)$ be an increasing sequence of positive numbers 
which goes to infinity sufficiently fast. ($a_n=n^2$ will do.)
Set 
\[ \Lambda := \mathbb{Z}^2\cap \left(\bigcup_{n=1}^\infty \{z\in \affine|\, |z-a_n|\leq n\}\right).\]
Let $c>0$. We define a meromorphic function $f(z)$ by 
\[ f(z) := \sum_{\lambda\in \Lambda}\frac{1}{(cz-\lambda)^3}.\]
We can choose $c$
so that $f \in \moduli(\affine P^1)$ and 
\[ \rho(f)>0, \quad \rho_{\mathrm{NSA}}(f) = \underline{\rho}_{\mathrm{NSA}}(f)=0.\]
\end{example}
For $f\in \moduli(X)$ we denote the closure of the $\affine$-orbit of $f$ by $\overline{\affine \cdot f}$.
Our main result is the following:
\begin{theorem} \label{thm: main theorem}
For any $f\in \moduli(X)$ we have 
\[ \rho(f) = \rho(\overline{\affine\cdot f})=\rho_{\mathrm{NSA}}(\overline{\affine\cdot f}) =
 \underline{\rho}_{\mathrm{NSA}}(\overline{\affine\cdot f}).\]
Hence for any $\affine$-invariant closed subset $\mathcal{N}\subset \moduli(X)$ 
\[ \rho(\mathcal{N})= \rho_{\mathrm{NSA}}(\mathcal{N})
   = \underline{\rho}_{\mathrm{NSA}}(\mathcal{N}). \]
\end{theorem}
The technique of the proof of Theorem \ref{thm: main theorem} also gives the following:
\begin{theorem} \label{thm: main theorem 2}
For any $\affine$-invariant closed subset $\mathcal{N}\subset \moduli(X)$
\begin{equation} \label{eq: statement of main theorem 2}
 \rho(\mathcal{N}) = \lim_{R\to \infty}\left(\frac{1}{\pi R^2}\sup_{f\in \mathcal{N}}\int_{|z|<R}|df|^2dxdy\right).
\end{equation}
\end{theorem}
The proofs of these theorems will be given in Section \ref{section: proofs of the main theorems}.
The essential ingredients of the proofs are the standard argument of normal family
(i.e. the compactness of $\moduli(X)$) and 
a technical result given in Section \ref{section: technical result}.

\section{Technical result.}{\label{section: technical result}
We fix a positive integer $D$ throughout this section. 
(Later we will need only the case $D=2$.)

We introduce one notation on Borel measures:
Let $\mu$ be a Borel measure on $\mathbb{R}^D$, and let $a\in \mathbb{R}^D$.
We define a Borel measure $a.\mu$ on $\mathbb{R}^D$ by 
$(a.\mu)(\Omega) := \mu(a+\Omega)$ where $\Omega\subset \mathbb{R}^D$ and
$a+\Omega:=\{a+x|\, x\in \Omega\}\subset \mathbb{R}^D$.

Let $\moduli$ be a set of Borel measures on $\mathbb{R}^D$ satisfying the following 
two conditions:
\begin{enumerate}
\item[(a)] For any $\mu\in \moduli$ and $a\in \mathbb{R}^D$ we have $a.\mu\in \moduli$.
\item[(b)] $\sup_{\mu\in \moduli} \mu([0,1]^D) < +\infty$.
\end{enumerate}
Under the condition (a), the condition (b) is equivalent to the condition that 
for every bounded Borel subset $\Omega\subset \mathbb{R}^D$ we have 
$\sup_{\mu\in \moduli}\mu(\Omega) < +\infty$.
\begin{example} \label{example: basic exmple of the set of measures}
Let $\varphi:\mathbb{R}^D\to [0,1]$ be a measurable function, and set 
\[ \mu(\Omega) := \int_\Omega \varphi d\mathrm{vol}, \quad (\Omega\subset \mathbb{R}^D).\]
Here $d\vol$ is the standard volume element of $\mathbb{R}^D$.
Then the set $\{a.\mu|\, a\in \mathbb{R}^D\}$ satisfies the above two conditions.
\end{example}
For a Borel set $\Omega\subset \mathbb{R}^D$ we denote its Lebesgue measure by $|\Omega|$.
For $r>0$ and $a\in \mathbb{R}^D$ we set $B_r(a):=\{x\in \mathbb{R}^D|\, |x-a|\leq r\}$.
We usually denote $B_r(0)$ by $B_r$.
We introduce the following two quantities:
\begin{equation*}
 \begin{split} 
  \rho &:= \lim_{R\to \infty} \left(\frac{1}{|B_R|}\sup_{\mu\in \moduli}\mu(B_R)\right),\\
  \tilde{\rho}&:=\lim_{r\to \infty}\left[\lim_{R\to\infty}\left\{\sup_{\mu\in \moduli}\left(\inf_{r\leq t\leq R}
                   \frac{\mu(B_t)}{|B_t|}\right)\right\}\right].
 \end{split}
\end{equation*}
The existence of the limit in the definition of $\rho$ follows from Lemma \ref{lemma: Ornstein-Weiss} below
(see the proof of Lemma \ref{lemma: consequence of OW}).
The quantity 
\[\sup_{\mu\in \moduli}\left(\inf_{r\leq t\leq R}
                   \frac{\mu(B_t)}{|B_t|}\right)\]
is a non-increasing function in $R$ and a non-decreasing function in $r$.
Hence the limits in the definition of $\tilde{\rho}$ exist.

The definition of $\tilde{\rho}$ looks complicated, but it is easy to see $\tilde{\rho}\leq \rho$.
The following result is the main technical tool for the proofs of Theorems \ref{thm: main theorem}
and \ref{thm: main theorem 2}.
\begin{theorem} \label{thm: technical theorem}
$\tilde{\rho}=\rho$.
\end{theorem}
This result might be known to some specialists in harmonic analysis or ergodic theory.
But I could not find a literature containing this result.

We need two lemmas below.
Lemma \ref{lemma: Vitali} is the well-known finite Vitali covering lemma 
(see e.g. Einsiedler-Ward \cite[p. 40, Lemma 2.27]{Einsiedler-Ward}).
Lemma \ref{lemma: Ornstein-Weiss} is a special case of Ornstein-Weiss's lemma.
(This formulation is due to Gromov \cite[p. 336]{Gromov}. 
The original argument was given in Ornstein-Weiss \cite[Chapter I, Sections 2 and 3]{Ornstein-Weiss}.)
\begin{lemma} \label{lemma: Vitali}
Let $a_1,\dots,a_K\in \mathbb{R}^D$ and $r_1,\dots,r_K>0$.
Then we can choose $1\leq i(1)<\dots<i(k)\leq K$ such that 
the balls $B_{r_{i(1)}}(a_{i(1)}), \dots, B_{r_{i(k)}}(a_{i(k)})$ are disjoint and 
\[ \bigcup_{j=1}^K B_{r_j}(a_j)\subset \bigcup_{j=1}^k B_{3r_{i(j)}}(a_{i(j)}).\]
\end{lemma}
Before giving the statement of Lemma \ref{lemma: Ornstein-Weiss} we need to prepare 
some terminologies.
Let $\Omega\subset \mathbb{R}^D$ and $r>0$. We define $\partial_r\Omega$ 
as the set of points $x\in \mathbb{R}^D$ such that $B_r(x)$ has a non-empty intersection both with $\Omega$ 
and $\mathbb{R}^D\setminus \Omega$.
A sequence of bounded Borel subsets $\{\Omega_n\}_{n\geq 1}$ of $\mathbb{R}^D$ is called 
a F{\o}lner sequence if for all $r>0$
\[ |\partial_r\Omega_n|/|\Omega_n| \to 0 \quad (n\to \infty).\]
The sequence $\{B_n\}_{n\geq 1}$ is a F{\o}lner sequence.
The sequence $\{[0,n]^D\}_{n\geq 1}$ is also.
\begin{lemma}\label{lemma: Ornstein-Weiss}
Let $h$ be a non-negative function on the set of bounded Borel subsets of $\mathbb{R}^D$ satisfying 
the following three conditions.

\noindent 
(Monotonicity) If $\Omega_1\subset \Omega_2$, then $h(\Omega_1)\leq h(\Omega_2)$.

\noindent 
(Subadditivity) $h(\Omega_1\cup \Omega_2) \leq h(\Omega_1)+h(\Omega_2)$.

\noindent 
(Invariance) 
For any $a\in \mathbb{R}^D$ and any bounded Borel subset $\Omega \subset \mathbb{R}^D$, we have 
$h(a+\Omega)=h(\Omega)$.

Then for any F{\o}lner sequence $\Omega_n$ $(n\geq 1)$ in $\mathbb{R}^D$, the limit of the sequence 
\[ h(\Omega_n)/|\Omega_n|\quad (n\geq 1) \]
exists, and its value is independent of the choice of a F{\o}lner sequence.
\end{lemma}
The following is an immediate consequence of Lemma \ref{lemma: Ornstein-Weiss}.
\begin{lemma} \label{lemma: consequence of OW}
For any $\varepsilon>0$ there exists $N=N(\varepsilon) >0$ such that 
every bounded Borel subset $\Omega\subset \mathbb{R}^D$ with $|\partial_N\Omega|/|\Omega|<1/N$
satisfies 
\[ \left|\frac{\sup_{\mu\in \moduli}\mu(\Omega)}{|\Omega|}-\rho\right| < \varepsilon.\]
\end{lemma}
\begin{proof}
Set $h(\Omega) := \sup_{\mu\in \moduli}\mu(\Omega)$.
This satisfies the three conditions in Lemma \ref{lemma: Ornstein-Weiss}.
If the above statement is false, then there exist $\varepsilon>0$ and a sequence of bounded Borel subsets 
$\Omega_n\subset \mathbb{R}^D$ with 
$|\partial_n\Omega_n|/|\Omega_n| < 1/n$ satisfying 
\[  \left|\frac{h(\Omega_n)}{|\Omega_n|}-\rho\right| \geq \varepsilon.\]
But $\Omega_n$ is a F{\o}lner sequence. So 
\[\rho = \lim_{n\to \infty}h(\Omega_n)/|\Omega_n|.\]
\end{proof}
\begin{proof}[Proof of Theorem \ref{thm: technical theorem}]
Assume $\tilde{\rho} < \rho-\delta$ for some $\delta>0$.
Set $\varepsilon := \delta/(2\cdot 3^{D+1})$.
Let $N=N(\varepsilon)$ be a positive number given by Lemma \ref{lemma: consequence of OW}.
We choose $r>0$ sufficiently large so that every $t\geq r$ satisfies 
\begin{equation}  \label{eq: choice of r}
 \frac{|\partial_N B_t|}{|B_t|} < \frac{1}{3N}.
\end{equation}
We fix $R>r$ so that 
\[ \sup_{\mu\in \moduli}\left(\inf_{r\leq t\leq R}\frac{\mu(B_t)}{|B_t|}\right) < \rho-\delta.\]
Let $L>R$ be a large number satisfying 
\begin{equation}  \label{eq: choice of L}
 |B_{L-R}|>\frac{|B_L|}{3}, \quad \left(\frac{1}{2}-\frac{1}{3^{D+1}}\right)|B_L| >|B_R|.
\end{equation}
Fix an arbitrary $\mu\in \moduli$.
For each $a\in \mathbb{R}$ there is $t=t(a)\in [r,R]$ such that 
\begin{equation}\label{eq: choice of t(a)}
  \frac{\mu(B_t(a))}{|B_t|} = \frac{(a.\mu)(B_t)}{|B_t|} < \rho-\delta.
\end{equation}
By the finite Vitali covering lemma (Lemma \ref{lemma: Vitali}),
we can choose $a_1,\dots,a_K\in B_{L-R}$ (set $t_i:=t(a_i)$) such that 
$B_{t_i}(a_i)\cap B_{t_j}(a_j)=\emptyset$ $(i\neq j)$ and
\[ B_{L-R}\subset \bigcup_{i=1}^KB_{3t_i}(a_i).\]
By the first condition of (\ref{eq: choice of L})
\[ 3^{-D-1}|B_L| < \sum_{i=1}^K|B_{t_i}(a_i)|.\]
Then we can choose (using the second condition of (\ref{eq: choice of L})) $1\leq J\leq K$ such that 
\begin{equation} \label{eq: choice of J}
  3^{-D-1}|B_L| < \sum_{i=1}^J|B_{t_i}(a_i)| \leq \frac{|B_L|}{2}.
\end{equation}
By (\ref{eq: choice of t(a)}) 
\begin{equation} \label{eq: critical decrease}
  \mu\left(\bigcup_{i=1}^JB_{t_i}(a_i)\right) < (\rho-\delta)\left|\bigcup_{i=1}^JB_{t_i}(a_i)\right|.
\end{equation}
Set $\Omega := B_L\setminus \bigcup_{i=1}^JB_{t_i}(a_i)$.
$|\Omega|\geq |B_L|/2$.
Since $\partial_N\Omega\subset \partial_NB_L\cup \bigcup_{i=1}^J\partial_NB_{t_i}(a_i)$,
\begin{equation*}
 \begin{split} 
   |\partial_N\Omega| &\leq |\partial_N B_L| +\sum_{i=1}^J|\partial_NB_{t_i}(a_i)| \\
   & < \frac{1}{3N}\left(|B_L|+\sum_{i=1}^J|B_{t_i}(a_i)|\right) \quad (\text{by (\ref{eq: choice of r}}))\\
   &\leq \frac{|B_L|}{2N} \leq \frac{|\Omega|}{N} \quad (\text{by (\ref{eq: choice of J}})).
 \end{split}
\end{equation*}
Hence by Lemma \ref{lemma: consequence of OW}
\[\frac{\mu(\Omega)}{|\Omega|} < \rho+\varepsilon.\]
So by (\ref{eq: critical decrease})
\begin{equation*}
 \begin{split}
   \mu(B_L) &= \mu(\Omega)+\mu\left(\bigcup_{i=1}^JB_{t_i}(a_i)\right) \\
   &< (\rho+\varepsilon)|\Omega| +(\rho-\delta)\left|\bigcup_{i=1}^JB_{t_i}(a_i)\right| \\
   &= \rho|B_L| + \underbrace{\left(\varepsilon |\Omega|-\delta\left|\bigcup_{i=1}^JB_{t_i}(a_i)\right|\right)}_A.
 \end{split}
\end{equation*}
By (\ref{eq: choice of J}) and $\varepsilon = \delta/(2\cdot 3^{D+1})$,
\[ A < \varepsilon |B_L|-\delta\cdot 3^{-D-1}|B_L| = -\varepsilon |B_L|.\]
Thus 
\[ \frac{\mu(B_L)}{|B_L|} < \rho -\varepsilon.\]
Since $\mu\in \moduli$ is arbitrary, 
\[ \frac{1}{|B_L|}\sup_{\mu\in \moduli}\mu(B_L) \leq \rho-\varepsilon.\]
We can let $L\to +\infty$. Hence $\rho\leq \rho-\varepsilon$.
This is a contradiction.
\end{proof}

\begin{remark}
In the above proof we have not used the complete additivity of measures $\mu\in \moduli$.
We needed only the monotonicity and subadditivity (two conditions given in Lemma 
\ref{lemma: Ornstein-Weiss}) of $\mu\in \moduli$.
So Theorem \ref{thm: technical theorem} can be also applied to 
a set of monotone, subadditive, non-negative functions on the set of bounded Borel 
subsets of $\mathbb{R}^D$ satisfying the
conditions (a) and (b) in the beginning of this section.
This generalization is not used in this paper.
But it might become useful in some future.
\end{remark}
Applying Theorem \ref{thm: technical theorem} to Example \ref{example: basic exmple of the set of measures},
we get the following corollary:
\begin{corollary}\label{cor: technical corollary}
Let $\varphi:\mathbb{R}^D\to [0,1]$ be a measurable function.
Then 
\[
  \lim_{R\to \infty}\left(\frac{1}{|B_R|}\sup_{a\in \mathbb{R}^D}\int_{B_R(a)}\varphi d\mathrm{vol}\right) \\
  = \lim_{r\to \infty}\left[\lim_{R\to\infty}\left\{\sup_{a\in \mathbb{R}^D}\left(\inf_{r\leq t\leq R}
                   \frac{\int_{B_t(a)}\varphi d\mathrm{vol}}{|B_t|}\right)\right\}\right].\]
\end{corollary}

\section{Proofs of Theorems \ref{thm: main theorem} and \ref{thm: main theorem 2}.}
\label{section: proofs of the main theorems}
Let $f:\affine\to X$ be a Brody curve.
We first prove Theorem \ref{thm: main theorem}.

\textbf{Step 1.} $\rho(f) = \rho(\overline{\affine\cdot f})$.
\begin{proof}
It is enough to prove that $\rho(g)\leq \rho(f)$ for all $g\in \overline{\affine\cdot f}$.
Take a sequence $\{a_n\}_{n\geq 1}\subset \affine$ such that $f(z+a_n)$ converges to $g(z)$ 
uniformly over every compact subset of $\affine$.
Let $\varepsilon>0$.
For any $R>0$ and $b\in \affine$ there exists $n_0>0$ such that 
for $n\geq n_0$ 
\[ \left||df|^2(z+a_n)-|dg|^2(z)\right| < \varepsilon \quad (|z-b|<R).\]
Hence for $n\geq n_0$
\begin{equation*}
 \begin{split}
  \frac{1}{\pi R^2}\int_{|z-b|<R}|dg|^2dxdy &\leq \frac{1}{\pi R^2}\int_{|z-a_n-b|<R}|df|^2(z) dxdy + \varepsilon\\
  &\leq \frac{1}{\pi R^2}\sup_{a\in\affine}\int_{|z-a|<R}|df|^2dxdy + \varepsilon.
 \end{split}
\end{equation*}
Taking the supremum with respect to $b$ and $R\to \infty$, we get $\rho(g)\leq \rho(f)+\varepsilon$.
Let $\varepsilon\to 0$. We get $\rho(g)\leq \rho(f)$.
\end{proof}
\textbf{Step 2.}
$\rho(f) = \underline{\rho}_{\mathrm{NSA}}(\overline{\affine\cdot f}) 
= \rho_{\mathrm{NSA}}(\overline{\affine\cdot f})$. 
(This completes the proof of Theorem \ref{thm: main theorem}.)
\begin{proof}
From Step 1, we get $\underline{\rho}_{\mathrm{NSA}}(\overline{\affine\cdot f}) \leq 
\rho_{\mathrm{NSA}}(\overline{\affine\cdot f}) \leq \rho(\overline{\affine\cdot f}) = \rho(f)$.
So it is enough to prove $\underline{\rho}_{\mathrm{NSA}}(\overline{\affine\cdot f})\geq \rho(f)$.
By Corollary \ref{cor: technical corollary} $\rho(f)$ is equal to 
\[ \lim_{r\to \infty}\left[\lim_{R\to\infty}\left\{\sup_{a\in \affine}\left(\inf_{r\leq t\leq R}
                   \frac{\int_{B_t(a)}|df|^2 dxdy}{\pi t^2}\right)\right\}\right].\]
Let $\varepsilon>0$ and fix $r=r(\varepsilon)>1$ satisfying
\[ \lim_{R\to\infty}\left\{\sup_{a\in \affine}\left(\inf_{r\leq t\leq R}
                   \frac{\int_{B_t(a)}|df|^2 dxdy}{\pi t^2}\right)\right\} > \rho(f)-\varepsilon.\]
Then for any $R>r$ there exists $a(R)\in \affine$ such that 
\[ \inf_{r\leq t\leq R}\frac{1}{\pi t^2}\int_{B_t(a(R))}|df|^2(z)dxdy > \rho(f)-\varepsilon.\]
Since $\moduli(X)$ is compact, we can take a sequence $r<R_1<R_2<R_3<\dots \to \infty$
(set $a_k := a(R_k)$) such that $f(z+a_k)$ converges to some $g(z)$ in $\moduli(X)$.
(Then $g\in \overline{\affine\cdot f}$.)
We have 
\[ \inf_{r\leq t\leq R_k} \frac{1}{\pi t^2}\int_{B_t} |df|^2(z+a_k)dxdy > \rho(f)-\varepsilon. \]
Hence for any $t\geq r$  we get 
\[ \frac{1}{\pi t^2} \int_{B_t} |dg|^2(z) dxdy \geq \rho(f)-\varepsilon.\]
Then for $s\geq r\, (>1)$
\begin{equation*}
  \begin{split}
   T(s,g)  &\geq \int_r^s\left(\int_{B_t}|dg|^2dxdy\right)\frac{dt}{t} \\
   &\geq (\rho(f)-\varepsilon)\left(\frac{\pi s^2}{2}-\frac{\pi r^2}{2}\right).
  \end{split}
\end{equation*}
Hence for $s\geq r$
\[ \frac{2}{\pi s^2} T(s,g) \geq (\rho(f)-\varepsilon)\left(1-\frac{r^2}{s^2}\right).\]
Taking the limit-inf with respect to $s$, we get 
$\underline{\rho}_{\mathrm{NSA}}(g)\geq \rho(f)-\varepsilon$.
Thus 
\[\underline{\rho}_{\mathrm{NSA}}(\overline{\affine \cdot f})
\geq\underline{\rho}_{\mathrm{NSA}}(g) \geq \rho(f)-\varepsilon.\]
$\varepsilon>0$ is arbitrary. 
So $\underline{\rho}_{\mathrm{NSA}}(\overline{\affine \cdot f}) \geq \rho(f)$.
\end{proof}
\begin{remark}
By using the above argument, we can also prove that $\rho(f)$ is equal to the supremum of 
\[ \limsup_{r\to +\infty}\left(\frac{1}{\pi r^2}\int_{|z|<r}|dg|^2 dxdy\right) \]
over $g\in \overline{\affine\cdot f}$.
(The limit-sup can be replaced with the limit-inf.)
This type of energy density was introduced and studied in \cite{Tsukamoto-packing}.
\end{remark}
\begin{proof}[Proof of Theorem \ref{thm: main theorem 2}]
Let $\rho$ be the right-hand-side of (\ref{eq: statement of main theorem 2}).
$\rho\geq \rho(\mathcal{N})$ is obvious (by the $\affine$-invariance of $\mathcal{N}$).
For each $f\in \mathcal{N}$ we define a Borel measure $\mu_f$ on $\affine$ by 
$\mu_f(\Omega) := \int_\Omega |df|^2dxdy$.
Consider the set $\{\mu_f|\, f\in \mathcal{N}\}$.
This set satisfies the conditions (a) and (b) in the beginning of Section \ref{section: technical result}.
Then Theorem \ref{thm: technical theorem} implies that $\rho$ is equal to 
\[ \lim_{r\to \infty}\left[\lim_{R\to\infty}\left\{\sup_{f\in \mathcal{N}}\left(\inf_{r\leq t\leq R}
                   \frac{\int_{B_t}|df|^2 dxdy}{\pi t^2}\right)\right\}\right].\]
Then, as in the proof of Step 2, for every $\varepsilon>0$ we can find $r_\varepsilon>0$ and 
$g_\varepsilon\in \mathcal{N}$ such that for all $t\geq r_\varepsilon$ 
\[ \frac{1}{\pi t^2}\int_{B_t}|dg_\varepsilon|^2dxdy\geq \rho-\varepsilon.\]
Then $\rho(\mathcal{N})\geq \rho(g_\varepsilon)\geq \rho-\varepsilon$.
Since $\varepsilon>0$ is arbitrary, we get $\rho(\mathcal{N})\geq \rho$.
\end{proof}


\textbf{Acknowledgements.}
M. Tsukamoto was supported by Grant-in-Aid for Young Scientists (B) (21740048)
from the Ministry of Education, Culture, Sports, Science and Technology.

\vspace{0.5cm}

\address{ Masaki Tsukamoto \endgraf
Department of Mathematics, Kyoto University, Kyoto 606-8502, Japan}

\textit{E-mail address}: \texttt{tukamoto@math.kyoto-u.ac.jp}


\begin{thebibliography}{100}





\bibitem{Brody}
R. Brody,
Compact manifolds and hyperbolicity,
Trans. Amer. Math. Soc. \textbf{235} (1978) 213-219


\bibitem{Einsiedler-Ward}
M. Einsiedler, T. Ward,
Ergodic theory with a view towards number theory,
Graduate Texts in Mathematics \textbf{259}, Springer, London


\bibitem{Gromov}
M. Gromov, 
Topological invariants of dynamical systems and spaces of holomorphic maps: I,
Math. Phys. Anal. Geom. \textbf{2} (1999) 323-415


\bibitem{Matsuo-Tsukamoto}
S. Matsuo, M. Tsukamoto,
Brody curves and mean dimension, 
preprint, arXiv: 1110.6014



\bibitem{Ornstein-Weiss}
D.S. Ornstein, B. Weiss,
Entropy and isomorphism theorems for actions of amenable groups,
J. Analyse Math. \textbf{48} (1987) 1-141



\bibitem{Tsukamoto-Nagoya}
M. Tsukamoto,
Moduli spaces of Brody curves, energy and mean dimension,
Nagoya Math. J. \textbf{192} (2008) 27-58



\bibitem{Tsukamoto-packing}
M. Tsukamoto,
A packing problem for holomorphic curves,
Nagoya Math. J. \textbf{194} (2009) 33-68


\end{thebibliography}
\end{document}